\documentclass[onethmnum,onefignum,final]{siamltex}

\usepackage{amsmath}
\usepackage{amssymb}
\usepackage{amsfonts}
\usepackage{mathrsfs}
\usepackage{bbm}
\usepackage{graphicx,color}
\usepackage{enumerate}

\renewcommand{\Re}{\operatorname*{Re}}
\renewcommand{\Im}{\operatorname*{Im}}
\newcommand{\real}{\operatorname*{Re}}

\def\mi{{\mathrm{i}\mkern1mu}}
\def\bone{\mathbbm{1}}
\def \veps{\varepsilon}
\def \z{\zeta}

\def \vfi{\varphi}
\def \bv{{\mathbf{b}}}
\def \cv{{\mathbf{c}}}
\def \Wv{{\mathbf{W}}}
\newcommand{\omegav}{\boldsymbol{\omega}}
\newcommand{\phiv}{\boldsymbol{\phi}}
\newcommand{\gv}{{\mathbf{g}}}
\newcommand{\ev}{{\mathbf{e}}}
\newcommand{\Ev}{{\mathbf{E}}}

\newcommand{\bR}{\mathbb{R}}

\newcommand{\IR}{\mathbb{R}}

\newcommand{\arccosh}{\operatorname{arccosh}}
\newcommand{\diam}{\operatorname{diam}}

\newtheorem{remark}[theorem]{Remark}
\newtheorem{assumption}[theorem]{Assumption}

\makeatletter\@addtoreset{equation}{section}\makeatother
\graphicspath{ {images/}{images}}

\begin{document}

\title{Fast and oblivious algorithms for dissipative and 2D wave equations}

\author{L. Banjai \thanks{The Maxwell Instute for Mathematical Sciences, School of Mathematical \& Computer Sciences, Heriot-Watt University, Edinburgh EH14 4AS, UK.
({\tt l.banjai@hw.ac.uk})} \and M. L\'opez-Fern\'andez\thanks{Dipartimento di Matematica Guido Castelnuovo, Sapienza Universit\`a di Roma, Piazzale Aldo Moro 5, 00185 Roma, Italy ({\tt lopez@mat.uniroma1.it}). Partially supported by the Spanish grants MTM2012-31298
and MTM2014-54710-P of the Ministerio de Economia y Competitividad.} \and A. Sch\"adle \thanks{Mathematisches Institut, Heinrich-Heine-Universit\"at, 40255 D\"usseldorf, Germany ({\tt schaedle@am.uni-duesseldorf.de})}}

\maketitle

\begin{abstract}
The use of time-domain boundary integral equations has proved very
effective and efficient for three dimensional acoustic and
electromagnetic wave equations. In even dimensions and when some
dissipation is present, time-domain boundary equations contain an
infinite memory tail. Due to this, computation for longer times
becomes exceedingly expensive. In this paper we show how oblivious
quadrature, initially designed for parabolic problems, can be used to
significantly reduce both the cost and the memory requirements of
computing this tail. We analyse Runge-Kutta based quadrature and
conclude the paper with numerical experiments.
\end{abstract}

\begin{keywords}
fast and oblivious algorithms, convolution quadrature, wave equations, 
boundary integral equations, retarded potentials, contour integral methods.
\end{keywords}

\begin{AMS}
65R20, 65L06, 65M15, 65M38
\end{AMS}

\section{Introduction}
Certain wave problems exhibit the property that behind the wave front
travelling at a finite speed there exists a smooth tail.  The simplest
examples of this phenomenon are the scalar wave equation in even
dimensions or damped wave equation in any dimension. Recently the
interest in the numerical solution of scalar and vector linear
wave equations by means of time-domain boundary integral equations has
risen sharply
\cite{stas_maxwell,BanKfast,DavD13,DavD14,LalS,LopS,LopS2,CheMWW,SauV13,SauV14,
Say13,MR3092588,MR3055890}. Two main approaches to the discretization
are time-space Galerkin methods \cite{BamH} and Convolution Quadrature
(CQ) \cite{Lub94}. 

The difficulty that arises in applying time-domain boundary integral
methods to dissipative or two-dimensional wave equation can best be
appreciated by comparing the free space Green's functions for the two
and three dimensional acoustic wave equation
\[
G_{\text{2D}} (t,x) = \frac{H(t-|x|)}{2\pi\sqrt{t^2-|x|^2}} \qquad
G_{\text{3D}} (t,x) = \frac{\delta(t-|x|)}{4\pi|x|},
\]
where $H(\cdot)$ is the Heaviside function and $\delta(\cdot)$ is the
Dirac delta. Here we see that whereas in three dimensions the support
of the Green's function is on the time-cone $t = |x|$, in two
dimensions the Green's function is non-zero for all $t > |x|$ and the
decay in $t$ is slow. This infinite tail forces an infinite memory on
time-domain boundary integral equation based methods that results in
expensive long-time computations. This will affect any
numerical method based on time-domain integral equations -- in
particular both time-space Galerkin and Convolution Quadrature. The
smoothness of this tail has been used in \cite{Mansur2d} and
\cite{BanG} to speed up computations and more importantly for the
current work, in \cite{BanG} it was noticed that this tail is due to
an operator of parabolic type; the precise meaning of this is
explained in the next section. This fact was used in \cite{BanG}
for purely theoretical purposes whereas in this paper it will be used
to develop a fast algorithm with reduced memory requirements. Our new
algorithm applied to the tail of the linear hyperbolic problems has
essentially the same properties as the algorithms developed in
\cite{LubS,SchLoLu06} for parabolic problems. To be more precise, let
$n_0 > \diam(\Omega)/h$ where $h$ is the time-step and $\Omega$ is the
scatterer; in other words $n_0 > |x|/h$ for all $|x|$ in the
above. Fast CQ algorithms for hyperbolic problems, see~\cite{Ban10},
can compute the first $n_0$ steps in $O(n_0 \log^2 n_0)$ or $O(n_0\log
n_0)$ time and $O(n_0)$ memory and history. Denoting by $N = n-n_0$
the number of time-step after $n_0$, in the following we discuss the
additional cost required for a target accuracy $\veps$.
\begin{itemize}
\item The computational complexity is $O\left(\log\left(\frac 1{\veps}\right)N\log (N)\right)$. 

\item The number of evaluations of the transfer operator $K$ is
reduced from $O(N)$ to $O\left(\log\left(\frac 1{\veps}\right)\log(N)
\right)$.

\item The memory requirements are reduced from $O(N)$ to $O\left(\log\left(\frac 1{\veps}\right) \log(N) \right)$.
\end{itemize}

The structure of the paper is as follows. In the next section we
consider an abstract setting that covers the motivating application
described above. Next, Runge-Kutta based Convolution Quadrature is
introduced. In the main part of the paper, the description and
analysis of an efficient scheme to compute the convolution weights is
described. The paper concludes with a more detailed description of the
time-domain boundary integral method and with the results of numerical
experiments.

\section{The abstract setting}

Let $K(\lambda,d) : U_\delta \times \mathbb{R}_{>0}\rightarrow
\mathbb{C}$ be analytic as a function of  $\lambda$ in the sector
\[
U_\delta = \{ \lambda \in \mathbb{C}\, : \, |\operatorname{Arg} \lambda| < \pi -\delta\},
\qquad \delta \in (0,\pi/2).
\]
and for some $\mu \in \mathbb{R}$ bounded as
\begin{equation}
  \label{eq:bound_sect}
  \left|e^{\lambda d}K(\lambda,d)\right| \leq C(d) |\lambda|^{\mu}, \qquad  \lambda \in U_\delta.
\end{equation}
Note that this in turn implies the standard bound  for hyperbolic operators
\begin{equation}
  \label{eq:bound_nsect}
  |K(\lambda,d)| \leq C(d) |\lambda|^{\mu}, \qquad \text{for }\Re \lambda > 0.
\end{equation}
We are interested in computing convolutions
\begin{equation}
  \label{eq:conv}
  u(t) = \int_0^tk(t-s,d) g(s)ds,
\end{equation}
where $k(t,d)$ is the inverse Laplace transform of $K(\lambda,d)$. If $\mu <-1$,
$k(t,d)$ as a function of $t$ is continuous and the above convolution is a
well-defined continuous function for integrable $g$.  For $\mu \geq -1$, the
convolution is defined directly via the inverse Laplace transform
\[
u(t)  = \frac{1}{2\pi\mi}\int_{\sigma+\mi\mathbb{R}} e^{\lambda t} K(\lambda,d) \mathscr{L}g(\lambda) d\lambda,
\]
where $\sigma > 0$ and $\mathscr{L}g (\lambda) = \int_0^\infty
e^{-\lambda t} g(t) dt$ denotes the Laplace transform.  For data $g$
such that its Laplace transform decays sufficiently fast, the inverse
Laplace transform above defines a continuous function, see
\cite{Lub94}.

Our aim is to describe an efficient algorithm for the computation of
(\ref{eq:conv}).  An important aspect of the algorithm is that it
should be effective for a range of $0 < d \leq R$ where $R > 0$ is
given.  So far in the literature, fast algorithms have been considered
in the non-sectorial, hyperbolic, case, i.e., kernels satisfying
(\ref{eq:bound_nsect}), \cite{BanS,Ban10}. In the sectorial,
parabolic, case, i.e., kernels bounded as
\begin{equation}
  \label{eq:bound_sectpure}
  \left|K(\lambda,d)\right| \leq C(d) |\lambda|^{\mu}, \qquad  \lambda \in U_\delta,
\end{equation}
fast algorithms which further allow huge memory savings are available
\cite{LubS,LopLPS,SchLoLu06,LoLuSch08}.  Our kernel is strictly
speaking non-sectorial, but after multiplication with $e^{\lambda d}$
becomes sectorial; this is what was meant by tail being due to a
parabolic operator.  Naturally, this special class of operators
requires its own fast algorithms.

\section{Runge-Kutta convolution quadrature}
Time discertization methods used in this paper, are based on
$A$-stable Runge-Kutta methods \cite{HaWII}. We employ standard
notation for an $s$-stage Runge-Kutta discretization based on the
Butcher tableau described by the matrix $\mathbf{A} =
(a_{ij})_{i,j=1}^s \in \IR^{s\times s}$ and the vectors $\bv =
(b_1,\ldots,b_s)^T \in \IR^s$ and $\cv = (c_1,\ldots,c_s)^T \in
[0,1]^s$. The corresponding stability function is given by
\begin{equation}\label{stabfun}
r(z) = 1+z\bv^T (\mathbf{I}-z \mathbf{A})^{-1} \bone,
\end{equation}
where
\[
\bone = (1,1,\dots,1)^T.
\]
Note that $A$-stability is equivalent to the condition $|r(z)| \leq 1$
for $\real z \leq 0$.  In the following we collect all the assumptions
on the Runge-Kutta method. These are satisfied by, for example, Radau
IIA and Lobatto IIIC families of Runge-Kutta methods.

\begin{assumption}\label{as:RK}
\begin{enumerate}[(a)]
%
\item
\label{as:RK-a}
The Runge-Kutta method is $A$-stable with (classical) order $p\ge 1$ and
stage order $q\leq p$.
\item
\label{as:RK-b}
The stability function satisfies
$|r(\mi y)| < 1$ for all real $y\ne 0$.
\item
\label{as:RK-c}
\label{strongly-A-stable}
$r(\infty)=0.$

\item
\label{as:RK-d} The Runge-Kutta coefficient matrix
$A$ is invertible.
\end{enumerate}
\end{assumption}
Since $r(z)$ is a rational function, the above assumptions imply that
\begin{equation}
  \label{eq:r_decay}
  r(z) = O(z^{-1}), \qquad |z| \rightarrow \infty.
\end{equation}

We define the weight matrices $\Wv_n$ corresponding to the operator $K$ by
\begin{equation}\label{Wn}
  \sum_{n=0}^\infty \Wv_n(d)\z^n = K\left( {\boldsymbol{\Delta}(\zeta) \over h},d\right),
\end{equation}
where
\begin{equation}\label{Delta}
  \boldsymbol{\Delta}(\zeta) = \Bigl(\mathbf{A} + {\z \over 1-\z}\bone \bv^T\Bigr)^{-1}.
\end{equation}
Denoting by $\omegav_n=(\omega_n^1,\dots,\omega_n^s)$ the last row of $\Wv_n$, the
approximation to the convolution integral \eqref{eq:conv} at time
$t_{n+1}=(n+1)h$ is given by
\begin{equation}\label{rk-cq}
 u_{n+1}=   \sum_{j=0}^n \sum_{i=1}^s  \omega_{n-j}^i(d)\, g( t_j + c_ih) =
   \sum_{j=0}^n  \omegav_{n-j}(d)\, \gv_j
\end{equation}
with the column vector $\gv_j = g(t_j+\mathbf{c}h)=\bigl( g(t_j+c_ih) \bigr)_{i=1}^s$.

The convergence order of this approximation has been investigated in
\cite{LubO93} for parabolic problems, i.e., for sectorial
$K(\lambda)$, and in \cite{BanL} and \cite{BanLM} for hyperbolic
problems, i.e., for non-sectorial operators.

With the row vector $\ev_n(z)=(e_n^1(z),\dots,e_n^s(z))$ defined as the last row
of the $s\times s$ matrix $\Ev_n(z)$ given by
\begin{equation}\label{En-def}
  (\Delta(\z)-zI)^{-1} = \sum_{n=0}^\infty \Ev_n(z) \,\z^n ,
\end{equation}
we obtain an integral formula for the weights
\begin{equation}\label{omegaRK-int}
  \omegav_n(d) = {h\over 2\pi \mi}
  \int_\Gamma  K(\lambda,d)\ev_n(h\lambda)\, d\lambda.
\end{equation}
This resepresentation follows from Cauchy's formula and the definition of the weights in \eqref{Wn}, with the integration contour $\Gamma$ chosen so that it
and surrounds the poles of $\ev_n(h\lambda)$. An explicit expressing for $\ev_n$ is  is given by
\begin{equation}\label{en-rk}
\ev_n(z) = r(z)^n \mathbf{q}(z),
\end{equation}
with the row vector $\mathbf{q}(z) =  \bv^T (I-z\mathbf{A})^{-1}$; cf.~\cite[Lemma~2.4]{LubO93}. The $A$-stability assumption implies that the poles of $r(z)$ are all in the right-half plane. Further, due to the decay of the rational function $r(z)$, see \eqref{eq:r_decay}, for $n > \mu+1$, the contour $\Gamma$ can be deformed into the imaginary axis.

For the weight matrices it holds
\begin{equation}\label{Wmatrix-int}
  \Wv_n(d) = {h\over 2\pi \mi}
  \int_\Gamma   K(\lambda,d) \Ev_n(h\lambda)\, d\lambda.
\end{equation}
By \cite[Lemma~2.4]{LubO93},  for $n\ge 1$, $\Ev_n(z)$ is the rank-1 matrix given
by
\begin{equation}\label{En-formula}
\Ev_n(z) = r(z)^{n-1} (I-z\mathbf{A})^{-1} \bone \bv^T (I-z\mathbf{A})^{-1}.
\end{equation}
The Runge-Kutta approximation of the inhomogeneous linear problem
\begin{equation}
  \label{eq:ode}
  y'(t) = \lambda y(t) + g(t), \quad y(0) = 0
\end{equation}
at time $t_{n+1}$ is given by
\begin{equation}
  \label{eq:rkode}
  y_{n+1}(\lambda) = h \sum_{j=0}^n \ev_{n-j}(h\lambda) \gv_j
\end{equation}
and thus the approximation of the convolution integral in~\eqref{rk-cq} can be rewritten as~\cite[Proposition~2.4]{LubO93}
\begin{equation}
  \label{eq:ci-int}
  u_{n+1} = \frac{1}{2\pi\mi} \int_\Gamma K(\lambda,d) y_{n+1}(\lambda) d\lambda.
\end{equation}
In the next section we discuss how to approximate the integral
in~\eqref{omegaRK-int} by an efficient quadrature rule.
It turns out that there exists $s>0$ such that for
$n_0 + s < n < n_0 + s B$ for an offset $n_0$ proportional to $d/h$
the quadrature error decays exponentially in the number of quadrature nodes.
The convergence rate depends on $B>1$, but is independent of $s$.
This is the key observation for the fast algorithm explained in
section~\ref{sec:fastalg}, where we split the sum in~\eqref{eq:rkode}
and use contour quadrature to approximate~\eqref{eq:ci-int}.

\section{Efficient quadrature for the computation of weights}\label{sec:quad}

If $K(\lambda)$ is a sectorial operator, in \cite{LopLPS} it is shown
that the contour in the integral formula for the weights $\omegav_n(d)$
in \eqref{omegaRK-int} can be chosen as the left branch of a hyperbola
with center at the origin and foci on the real axis. In
\cite{SchLoLu06} both hyperbolas and Talbot contours are tested and
shown to work in practice. In the case of kernels considered here, the
contour must not have a real part extending to
$-\infty$. Proposition~\ref{propn:trunc_error} shows that cutting the
hyperbola at a finite real part, commits a small error. For the proof
we will require the following technical lemma.

\begin{lemma}\label{lemma:gamma}
  Let
\[
\gamma(\xi) = \inf_{-\xi \leq \Re z \leq 0} \frac{\log |r(z)|}{\Re z}.
\]

Then $\gamma(\xi) \in (0,1]$ for  $\xi > 0$, it monotonically increases as $\xi \to 0$ and
\[
|r(z)| \leq |e^{\gamma(\xi) z}|,
\]
for all $z$ in the strip  $-\xi \leq \Re z \leq 0$.

Further, there exists $\rho > 0$ such that
 \[
 |r(z)| \leq |e^{2 z}|,
 \]
 for all $z$ in the strip  $0 \leq \Re z \leq \rho$.
\end{lemma}
\begin{proof}
By assumption $|r(z)| < 1$ for all $\Re z < 0$ and $r(\infty) = 0$,
hence $\gamma(\xi) >0$. Further, $\gamma(\xi)$ cannot be greater than
$1$, since this would contradict the approximation property
\[
r(z) = e^z + O(z^{p+1}).
\]
The bound on $r(z)$ is clear by the definition of $\gamma(\xi)$.

The remaining statement can be proved similarly, making sure that
$\rho$ is less than the real part of any singularity of $r(z)$; for a
similar statement see Lemma~1 in \cite{LopLPS}.
\end{proof}

\begin{remark}
Some numerically obtained values of $\gamma$ are given in
Table~\ref{tab:gamma}. For backward Euler $\gamma(\xi)$ is given explicitly by
\[
\gamma(\xi) = \frac{\log(1+\xi)}{\xi}.
\]
In order to reduce the number of constants in
the estimates, we have chosen not to be as precise about the bound for
the case $\Re z > 0$ . The optimality of the estimates has
nevertheless not been significantly affected.
\end{remark}

\begin{table}
  \centering
  \begin{tabular}{|c|c|c|c|}
    \hline
    & Backward Euler & 2-stage Radau IIA & 3-stage Radau IIA\\
    \hline
    $\xi = 1$ & 0.69 & 0.90 & 0.94\\
    $\xi = 1/2$ & 0.811 & 0.984 & 0.997\\
    \hline
  \end{tabular}
  \caption{\small Numerically obtained values of $\gamma(\xi)$ from Lemma~\ref{lemma:gamma}
    for different values of $\xi$.}
  \label{tab:gamma}
\end{table}

With this we have that the integrand in \eqref{omegaRK-int} can be bounded as
\begin{equation}
  \label{eq:exp_decay}
  |K(\lambda,d)r(\lambda h)^n| \leq C(d) |\lambda|^\mu |e^{-\lambda d/n}r(\lambda h)|^n
  \leq C(d) |\lambda|^\mu e^{\Re \lambda t_n(\gamma(\xi)-d/t_n)}
\end{equation}
and thus it decreases exponentially with $-\Re \lambda t_n$ as long as
$0 < -\Re \lambda h < \xi$ and $d/t_n <\gamma(\xi)$. This suggests
replacing contour $\Gamma$ in \eqref{omegaRK-int} with a finite
section of a hyperbola:
\begin{equation}
  \label{eq:hyperbola_finite}
  \Gamma_0 = \nu \varphi([-a,a]), \qquad \varphi(x) = 1-\sin(\alpha-\mi x), \quad \nu > 0.
\end{equation}
We want the endpoint of the finite hyperbola to be in the left-half complex plane, i.e.,
\[
\Re \varphi(a) = (1-\sin\alpha\cosh a) < 0 \quad \iff \quad \cosh a > 1/\sin\alpha.
\]
The right-most point on the hyperbola is given by
\begin{equation}
  \label{eq:rightmost}
  \nu \sup_{x \in [-a,a]} \Re \varphi(x) = \nu \varphi(0) = \nu(1-\sin\alpha) < \nu.
\end{equation}

The error commited in replacing $\Gamma$ with $\Gamma_0$ is investigated next.

\begin{proposition}\label{propn:trunc_error}
  Let $d, \delta>0$ and $\mu$ be given such that
  $K(\lambda,d)$ satisfies
  \eqref{eq:bound_sect} and let $a$ and $\alpha \in (0,\pi/2-\delta)$ be
  given such that $\cosh a > 1/\sin \alpha$.  Then for $h$, $\nu_0>0$,
  $m > \mu$, $0 < \nu \leq \nu_0$ and $t_{n-m} > d/\gamma(\xi)$ with
  $\xi = h \nu_0 |\Re \varphi(a)|$,
  \[
  \left\lVert\omegav_n(d) - \frac{h}{2\pi\mi}\int_{\Gamma_0} K(\lambda,d) \ev_n(\lambda h)  \, d\lambda\right\rVert
  \leq C  |\nu \varphi(a)|^{\mu-m} h^{-m} e^{\nu \Re \varphi(a)(\gamma(\xi) t_{n-m}-d)},
  \]
  where  $C = \text{const} \cdot C(d)$, {with $C(d)$ in \eqref{eq:bound_sect}}.
\end{proposition}
\begin{proof}
Let us choose the contour in \eqref{omegaRK-int} as $\Gamma =
\Gamma_{-1}+\Gamma_0+\Gamma_{1}$, where $\Gamma_0$ is defined by
\eqref{eq:hyperbola_finite}, $\Gamma_{-1} = \{ \bar w-\mi \varrho
\,|\, \varrho \in [0,\infty)\}$ and $\Gamma_{1} = \{ w+\mi \varrho
\,|\, \varrho \in [0,\infty)\}$, where $w = \nu \varphi(a)$; see
Figure~\ref{fig:contours}.

To prove the required result we need to bound
\[
\frac{h}{2\pi\mi}\int_{\Gamma_{\pm 1}} K(\lambda,d) \ev_n(\lambda h)  \, d\lambda.
\]

\begin{figure}
  \centering
  \includegraphics[width=.4\textwidth]{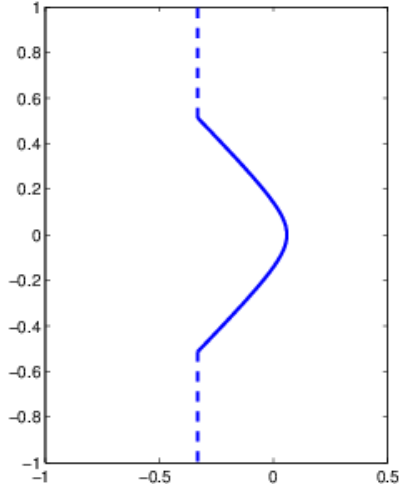}
  \caption{\small Contour $\Gamma_0$ is depicted by the solid line and  $\Gamma_{-1}$ and $\Gamma_1$ by the dashed lines.
  \label{fig:contours}}
\end{figure}

From the definition of $\Gamma_{\pm 1}$, \eqref{eq:exp_decay},
property \eqref{eq:bound_sect} of $K(\lambda)$ and the fact that $r(z)
= O(z^{-1})$ it follows that for $\lambda \in \Gamma_{-1} {\cup}
\Gamma_1$
\begin{equation}
  \label{eq:bound_rho1_c}
  \begin{split}
h  |K(\lambda,d)r(\lambda h)^n| &\leq C h^{1-m}|\lambda|^{\mu-m}e^{\Re\lambda(\gamma t_{n-m}-d)}\\
&= C h^{1-m}|\lambda|^{\mu-m}e^{\nu\Re \varphi(a)(\gamma t_{n-m}-d)}.
  \end{split}
\end{equation}

Hence
\[
\begin{split}
\lefteqn{h\int_{\Gamma_{\pm1}}\left\lVert K(\lambda, d)
r(\lambda h)^n \bv^T(I-\lambda h \mathbf{A}v)^{-1}\right\rVert}& \\
&\leq C h^{1-m} e^{\nu \Re \varphi(a)(\gamma t_{n-m}-d)} \int_{\Gamma_{\pm1}} |\lambda|^{\mu-m}\|(I-\lambda h\mathbf{A})^{-1}\|d\lambda \\
& \leq Ch^{-m} e^{\nu \Re \varphi(a)(\gamma t_{n-m}-d)}  \int_{\Gamma_{\pm1}} |\lambda|^{\mu-m-1} d\lambda \\
& \leq C |\nu \varphi(a)|^{\mu-m} h^{-m} e^{\nu \Re \varphi(a)(\gamma t_{n-m}-d)}.
\end{split}
\]
finishes the proof.
\end{proof}

The next step is to define a quadrature on the interval $[-a,a]$ and
hence on the contour $\Gamma_0$.  Since the integrand is small at the
edges of the interval, trapezoidal quadrature is a good choice and the
following result gives the quadrature error.
\begin{proposition}\label{propn:quad_error}
 Let $d, \delta>0$ and $\mu$ be given such that $K(\lambda,d)$
satisfies \eqref{eq:bound_sect}.  Let $\alpha \in (0,\pi/2-\delta)$
and $a>0$ be such that $\cosh a > 1/\sin \alpha$ and $0< b <
\min(\alpha, \pi/2-(\delta+\alpha))$. Assume that $h$ and $\nu_0>0$
are small enough so that $h \nu_0 (1-\sin(\alpha - b)) < \rho$ with
$\rho$ as in Lemma~\ref{lemma:gamma}.  For $L\in \mathbb{N}$ and $\tau
= a/L$ let
\begin{equation}\label{eq:xi}
\xi = h \nu_0 \sup_{y \in [-b,b]}|\Re \varphi(a+\tau/2+\mi y)|
 = -h \nu_0 (1-\sin(\alpha+b)\cosh (a+\tau/2)).
\end{equation}
Then for
\[
\mathbf{f}(x) = \frac{{\nu}h}{2\pi\mi}K(\nu\varphi(x))\ev_n(\nu\varphi(x)h) \varphi'(x)
\]
and
\[
\mathbf{I} = \int_{-a}^a \mathbf{f}(x) dx, \qquad \mathbf{I}_L = \frac{a}{L}{\sum_{k = -L}^L}\mathbf{f}(x_k),
\]
where $x_k = k\tau$, and any $0 < \nu < \nu_0$ it holds that
\[
\lVert \mathbf{I}-\mathbf{I}_L\rVert\leq C \left( h\frac{e^{2 t_n \nu} } {e^{2\pi b/\tau}-1}+
 \left(h\nu \cosh(a+\tau/2)\right)^{-m}e^{\nu \real \vfi(a+\tau/2 - \mi b)(\gamma(\xi) t_{n-m}-d)}\,\tau \right),
\]
for $m$ the smallest integer with  $m > \mu$, and
$$
C=C(d,\sin(\alpha+b)) \nu^{1+\mu} \max\{1,\left(\cosh(a+\tau/2)\right)^{1+\mu})\}.
$$

\end{proposition}
\begin{proof}
  Let us first suppose that $\mathbf{f}$ is analytic and bounded as
  $\lVert \mathbf{f}(z)\rVert \leq M$ for $z \in R_\tau = \{w \in \mathbb{C} : -a-\tau \leq \Re w \leq a+\tau$, $-b < \Im w < b\}$. Then it follows from Lemma~\ref{lemma:trapez} in the Appendix that
\begin{equation}\label{genbound}
\left\lVert\mathbf{I}-\mathbf{I}_L\right\rVert \leq\frac{2M}{e^{2\pi b/\tau}-1} +
\frac{\log 2}{\pi} \tau \sup_{y \in [-b,b]}\lVert\mathbf{f}(a+\tau/2+\mi y)-\mathbf{f}(-a-\tau/2+\mi y)\rVert.
\end{equation}
To finish the proof we need to first bound $\mathbf{f}$ in the rectangles $R_{\tau}$ and $R_{\tau/2}$, in a similar fashion as in \cite{LoPa}.

By the assumptions on $K$, for $b$ such that $0 < b <
\min(\alpha,\pi/2-(\delta+\alpha))$, the integrand is analytic in the
rectangle $R_{\tau}$. For any $x,y\in \bR$
$$
|\vfi(x+\mi y)| = \cosh x - \sin(\alpha+y) \quad \mbox{and} \quad
|\vfi'(x+\mi y)| = \sqrt{\cosh^2 x - \sin^2(\alpha+y)},
$$
leading to the following estimates for $z = x+\mi y \in R_{\tau}$
\begin{equation}\label{bounds_vfi}
  \begin{split}
    1-\sin(\alpha+b)  &\le |\vfi(x+\mi y)| \le \cosh (a+\tau) - \sin(\alpha-b), \\
    |\vfi'(x+\mi y)| &\le\sqrt{\cosh^2 (a+\tau) - \sin^2(\alpha-b)},\\
    \displaystyle \left|\frac{\vfi'(x+\mi y)}{\vfi(x+ \mi y)}\right| &\le
    \sqrt{\frac{1+\sin(\alpha+b)}{1-\sin(\alpha+b)}}.
  \end{split}
\end{equation}

Thus, if $\mu \le -1$ in \eqref{eq:bound_nsect}, by using the second
part of Lemma~\ref{lemma:gamma} we can bound the integrand for $z\in
R_{\tau}$ (and thus the constant $M$ in \eqref{genbound}) as
\[
\lVert|\mathbf{f}(z)\rVert
\leq \frac{h}{2\pi} C(d) \left(1+\sin(\alpha+b)\right)^{1/2} \left(1-\sin(\alpha+b) \right)^{1/2+\mu} \nu^{1+\mu} e^{2\nu t_n}.
\]
Using now that $\real\vfi(\pm a \pm \tau/2 +\mi y) < 0$ we can estimate as in Proposition~\ref{propn:trunc_error}, with $m=0$,
\begin{align*}
  \lVert\mathbf{f}(\pm a\pm \tau/2+\mi y)\rVert
  \leq  \frac{1}{2\pi}  C(d) & \left(1+\sin(\alpha+b)\right)^{1/2} \left(1-\sin(\alpha+b) \right)^{1/2+\mu} \cdot \\
  & \cdot \nu^{1+\mu} e^{\nu t_{n}\real \vfi(a+\tau/2 -\mi b) \left(\gamma(\xi)-d/t_{n}\right)}.
\end{align*}

For $\mu > -1$, we obtain instead
\[
\lVert\mathbf{f}(z) \rVert \leq  \frac{h}{2\pi} C(d) \sqrt{\frac{1+\sin(\alpha+b)}{1-\sin(\alpha+b)}} \nu^{1+\mu} \left(\cosh(a+\tau/2)\right)^{1+\mu}  e^{2\nu t_n}
\]
and
\begin{align*}
\lVert\mathbf{f}(\pm a\pm \tau/2+\mi y)\rVert
 \leq \frac{h^{-m}}{2\pi} C(d) & \sqrt{\frac{1+\sin(\alpha+b)}{1-\sin(\alpha+b)}} \cdot \\
 \cdot \left(\nu\cosh(a+\tau/2) \right)^{1+\mu-m} & e^{\nu t_{n-m }\real \vfi(a+\tau/2 -\mi b) \left(\gamma(\xi)-d/t_{n-m}\right)}.
\end{align*}
This finishes the proof.
\end{proof}
Combining the two propositions gives the following theorem.
\begin{theorem}\label{thm:main}
Under the conditions of Proposition~\ref{propn:quad_error} and with the same definition of $\mathbf{I}_L$
\begin{equation}\label{mainerr}
  \begin{split}
   &\lVert \omegav_n(d) -\mathbf{I}_L\rVert \leq \\
   & C \left(h\frac{e^{2 t_n
          \nu}}{e^{2\pi b/\tau}-1}+ (1+\tau)\left(h\nu \cosh a
      \right)^{-m} e^{\nu \left(1-\sin (\alpha-b) \cosh
          a\right)(\gamma(\xi) t_{n-m}-d) } \right),
  \end{split}
\end{equation}
with $m = \lceil \mu\rceil$ and
$$
C=C(d,\sin(\alpha+b)) \nu^{1+\mu} \max\{1,\left(\cosh(a+\tau/2)\right)^{1+\mu})\}.
$$
\end{theorem}

\subsection{A non-optimal choice of parameters}\label{sec:parameter_choice1}

We will first show that a good choice of the parameters exists resulting in an efficient algorithm. The optimal choice of parameters is discussed later.

To set the stage let us deal with the first term in the error estimate of Theorem~4.5 in a way that is standard for oblivious quadrature algorithms. Here time is split into ever increasing intervals. A novelty is is that we require the first interval to start at some $n_0h > d$.

Therefore, let $t_n =nh \in
\left[hn_0+ hB^\ell,hn_0 + h 2B^{\ell+1} \right]$, $\ell \geq 0$,  and denote by $T_\ell$ the right-end point of this interval, i.e.,
$
T_\ell = hn_0 + h 2B^{\ell+1}.
$
Choose $\nu_\ell = \frac{c_0}{T_\ell}$ and $a_\ell = c_1 \log T_\ell$ for some constants $B > 1$, $c_0 > 0$, $c_1 > 0$. For $\tau$ small enough, i.e., $L = a/\tau$ big enough, the first term can be made arbitrarily small. In fact  for the first term to be smaller than $\varepsilon_1$ we need
\[
L \propto \log T_\ell \log \frac1\varepsilon_1.
\]

To simplify the details of the analysis of the second term, let us assume $m = 0$ in Theorem~4.5 and let $\xi$ be given by the formula
\[
\xi(\nu_{\ell} ,a_{\ell} ) = h \nu_{\ell} |\Re \varphi(a_{\ell} )| = h \nu_{\ell}  (\cosh a_{\ell}  \sin \alpha -1)
\sim h,
\]
where in the last step we used $\nu_{\ell}  \sim \frac1{T_\ell}$ and $a_{\ell}  \sim \log T_\ell$.
In fact $\xi$ is given by a somewhat more complicated formula, but the inclusion of all the details would not change the asymptotic results we give here; for an optimal choice of parameters these details are of importance.
We can thus write the second term in the estimate of Theorem~4.5 as
\[
\varepsilon_2 = e^{-\frac{\xi(\nu_{\ell} ,a_{\ell} )}{ h}(\gamma(\xi(\nu_{\ell} ,a_{\ell} )) t_n-d) }
\sim e^{-(\gamma(h) t_n-d) }.
\]
Therefore as $L$ is increased we expect the error to decrease exponentially fast until it reaches $\varepsilon_2$; see Figure~\ref{fig:convergenceConttw}. Note that if we make $n_0$ large enough, this error can also be controlled.  For later intervals, i.e., for larger $\ell$, the second term quickly becomes insignificantly small.

\subsection{The fast and oblivious algorithm}
\label{sec:fastalg}
The algorithm here is similar to the fast and oblivious algorithm
described in~\cite{SchLoLu06}, but with a shift by $ n_0 = \lceil d/(h\gamma(\xi(\nu_0,a_0))) \rceil $ as explained above.

For $N_\ell$ the smallest integer such that
$n < n_0+1+B+\sum_{\ell=2}^{N_\ell}B^\ell$ the convolution~\eqref{rk-cq}
is split into $N_\ell$ sums
\begin{equation}
\label{eq:timesplit}
u_{n+1} = u_{n+1}^{(0)} + \sum_{\ell = 2}^{N_\ell} u_{n+1}^{(\ell)}
\end{equation}
where for suitable $b_\ell$ given below
$$
u_{n+1}^{(0)} = \sum_{j=b_1}^n \omegav_{n-j} \gv_j \mbox{ and }
u_{n+1}^{(\ell)} = \sum_{j=b_\ell}^{b_{\ell-1}-1} \omegav_{n-j} \gv_j.
$$
In view of the discussion in Section~\ref{sec:parameter_choice1} the $b_\ell$ are chosen such that
for $j = b_\ell, b_\ell+1,  \dots, b_{\ell-1}-1$ we have $t_{n-j}=(n-j)h \in
\left[hn_0+ hB^\ell,hn_0 + h 2B^{\ell+1} \right]$, where $n_0$ is a fixed integer offset with
$n_0\ge d/(h \gamma(\xi))$. Inserting \eqref{omegaRK-int} and using \eqref{en-rk} we
obtain
\begin{align}
  u_{n+1}^{(\ell)} &= \sum_{j=b_\ell}^{b_{\ell-1}-1} {h\over 2\pi \mi}
  \int_\Gamma K(\lambda,d)\mathbf{e}_{n-j}(h\lambda)\, d\lambda\, \gv_j\label{eq:ull}
  \\
  &= \frac{1}{2\pi\mi} \int_\Gamma r(h\lambda)^{n-(b_{\ell-1}-1)} K(\lambda,d) y^{(\ell)}(h\lambda) \, d\lambda
\end{align}
with $y^{(\ell)}(\lambda;b_{\ell-1},b_\ell)$, similar to
\eqref{eq:rkode}, the Runge-Kutta approximation to \eqref{eq:ode} at
time $t=hb_{\ell-1}$ with initial value $y^{(\ell)}(hb_\ell) = 0$.
The contour integral $\int_\Gamma$ is approximated by the
$\ell$-dependent approximation given in Proposition
\ref{propn:quad_error}. Instead of keeping all the history, i.e.\ $\gv_j$
for $j=0,\dots,n$ in memory, for evaluating the convolution the
algorithm requires only three copies of the Runge-Kutta solution
$y^{(\ell)}(\lambda;b_\ell,b_{\ell-1})$ for $\ell=2\dots N_\ell$ and
each $\lambda^{(\ell)}_k = \nu_\ell \vfi(\tau_\ell)$, which are
calculated step by step. Details can be found in~\cite{SchLoLu06}.  As
$N_\ell$ is proportional to $\log(n-n_0)$, the memory requirement thus
grows like $\mathcal{O}(\log(n-n_0))$ and the operation count as
$\mathcal{O}((n-n_0)\log(n-n_0))$. The computation of $u_{n+1}^{(0)}$ is done with standard CQ algorithms  with the number of evaluations of the kernel and memory requirements  proportional to $n_0$.

\subsection{Optimal choice of parameters}
We begin with a corollary of Theorem~\ref{thm:main}.
\begin{corollary}\label{coro:para}
Assume that $t_n \in [t_0,\Lambda t_0]$, for given $t_0>0$ and $\Lambda \ge 1$, and that there exists $0<D<1$ such that $d \leq D \gamma(\xi)t_0$.
Then for every $\theta \in (0,1)$, the following choice of parameters
$$
\tau = \frac{a(\theta)}{L}  \qquad \nu = \frac{\pi b L \theta}{\Lambda t_0 a(\theta)}.
$$
with
$$
a(\theta)=\arccosh \left(  \frac{\gamma(\xi)(1-D)\theta + 2\Lambda(1-\theta)}{\gamma(\xi)(1-D) \theta \sin(\alpha-b)} \right)
$$
yields the uniform error estimate
$$
|\omegav_n(d) -\mathbf{I}_L| \leq  C \exp\left(-\frac{2\pi b L (1-\theta)}{a(\theta)}\right),
$$
where $C$ includes all non exponentially growing terms in the bound \eqref{mainerr}.
\end{corollary}
The above choice of parameters is independent of $h$.
\begin{proof}
The stated choice for $\tau$ and $\nu$ guarantees that exponents in the bound \eqref{mainerr} with $D$ in place of $d$ are equal and negative, with
$$
2\Lambda \nu t_0 = \theta \frac{2\pi b}{\tau}, \quad \theta \in(0,1)
$$
and
$$
(\theta-1) \frac{2\pi b}{\tau}= \gamma \nu t_0 \left(1-\sin (\alpha-b) \cosh a\right).
$$
\end{proof}
\begin{remark}
Notice that $\gamma(\xi)$ and $a(\theta)$ are nonlinearly related via
\eqref{eq:xi}. In our experiments we have opted to fix the value of
$\gamma(\xi)\in(0,1)$ and then use the error estimate in
Corollary~\ref{coro:para} to compute $\theta$, $a(\theta)$ and
$\tau(\theta)$. This strategy may lead to an underestimation of the
stability function according to Lemma~\ref{lemma:gamma}. Still, our
numerical results show that reasonable values of $\gamma$ and good
choices for the rest of parameters are easy to find for prescribed
accuracies. These parameters depend on $\delta, \mu$ and $d$ in
\eqref{eq:bound_sect} but not on the particular application of our
method. Results for different values of $\gamma(\xi)$ are displayed in
Figures~\ref{fig:evolutionerror} and \ref{fig:convergenceCont}, for a
particular example.
\end{remark}

The effect of round-off errors can be included in the analysis in the
same way as in \cite{LoPaScha}, leading to a minimization problem for
the choice of $\theta$. In the simplest case of analysis in
\cite{LoPaScha}, the propagation of the errors in the evaluation of
$K$, that we denote by $\varepsilon$, is governed by the exponentially
growing term
$$
\varepsilon e^{\Lambda \sigma} = \varepsilon \epsilon(\theta)^{-\theta/2}, \quad \mbox{ with }\quad \epsilon(\theta)= e^{-2\pi b L/a(\theta)}.
$$
Then from Corollary~\ref{coro:para} we deduce that the total error estimate is of the form
\begin{equation}\label{bound}
\varepsilon \epsilon(\theta)^{-\theta/2} + \epsilon(\theta)^{1-\theta}.
\end{equation}
The choice $\theta=1/L$ above guarantees the boundedness of the term in $\varepsilon$ and the control of the error propagation, giving a convergence rate like $O(e^{-cL/\ln L})$. A better choice of $\theta$ can be obtained by minimizing \eqref{bound} for given $L$, $\alpha$, $b$, and $\Lambda$.

\subsection{A numerical experiment}
We illustrate Theorem~\ref{thm:main} and Corollary~\ref{coro:para} by
considering in \eqref{eq:ull} the generating function
\[
K(\lambda,d) = K_0(\lambda d)
\]
where $K_0(\cdot)$ is a modified Bessel function \cite{AbrS}.  This
function satisfies the bounds (\ref{eq:bound_nsect}) and
(\ref{eq:bound_sect}) with $\mu = -1/2$ as proved in \cite{BanG}. In
Figures~\ref{fig:scalar_dbig} and \ref{fig:scalar_dsmall} we show the
error in the approximation of the convolution weights in \eqref{eq:ull}
along approximation intervals of the form $[hn_0+hB^{\ell},hn_0+2hB^{\ell+1}]$, with
$B=5$ and $B=10$ and for two different values of the distance parameter $d$,
namely $d=0.1$ (Fig. \ref{fig:scalar_dbig}) and $d=0.01$ (Fig. \ref{fig:scalar_dsmall}).

The lower row of error curves corresponds to the case $B=5$, where we
take $L=15$ quadrature nodes on the hyperbola and consider seven
approximation intervals, i.e. $\ell = 2,\dots,8$.  The upper row of
error curves, with the larger error corresponds to $B=10$, $L=10$ and
$\ell = 2,\dots,6$. The other parameters that determine the hyperbola
are as follows $\alpha = 0.9$, $\gamma(\xi)=0.8$, $b = 0.6$ and $n_0 =
\lceil d/\gamma(\xi)/h \rceil$.
\begin{figure}
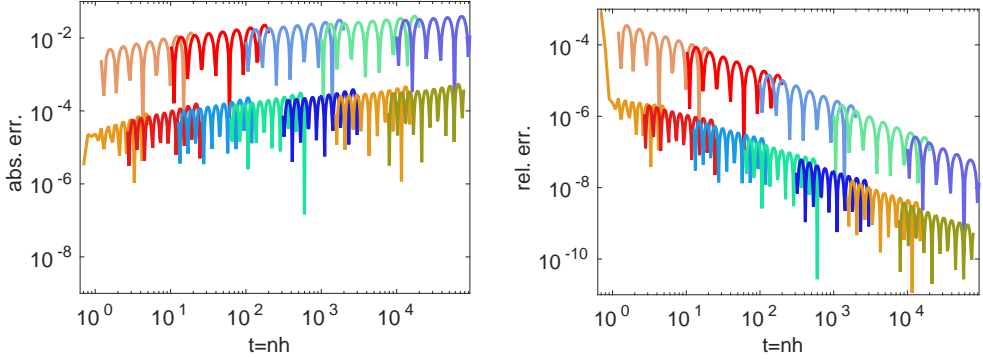

  \centering
  \includegraphics[width=.48\textwidth]{ANewHyperberrN10N15_dis_RKVRadauIIA5besselk0_d0_1}\hfill
  \includegraphics[width=.48\textwidth]{NewHyperberrN10N15_dis_RKVRadauIIA5besselk0_d0_1}
  \caption{\small Absolute (left) and relative (right) errors in the computation
    of weights $\omega_n(d)$ for the 3 stage Radau IIA method of order 5
    for $d = 0.1$, number of quadrature points $L = 10$ and basis $B = 10$ (top row)
    and $L=15$ and $B=5$ (bottom row).
  }
  \label{fig:scalar_dbig}
\end{figure}
\begin{figure}
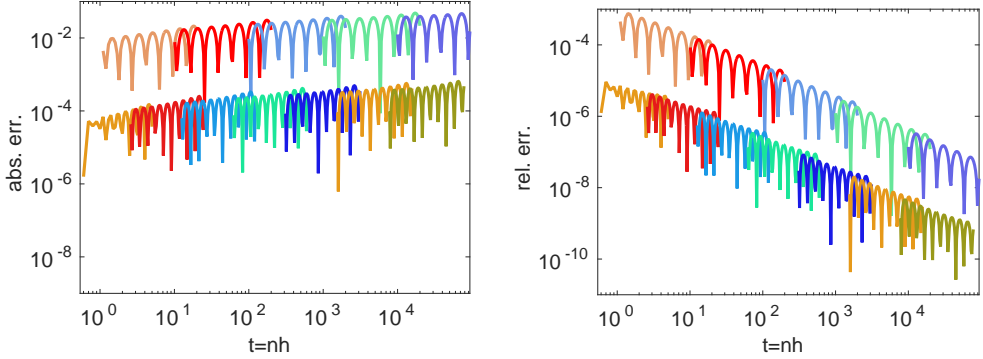

  \centering
  \includegraphics[width=.48\textwidth]{ANewHyperberrN10N15_dis_RKVRadauIIA5besselk0_d0_01}\hfill
  \includegraphics[width=.48\textwidth]{NewHyperberrN10N15_dis_RKVRadauIIA5besselk0_d0_01}
  \caption{\small Absolute (left) and relative (right) errors in the computation
    of weights $\omega_n(d)$ for the 3 stage Radau IIA method of order 5
    for $d = 0.01$, number of quadrature points $L = 10$ and basis $B = 10$ (top row)
    and $L=15$ and $B=5$ (bottom row).}
  \label{fig:scalar_dsmall}
\end{figure}

\section{An application}

As explained in the introduction, the initial motivation for
investigating operators satisfying (\ref{eq:bound_sect}) comes from
the application of time-domain boundary integral operators for wave
propagation in cases where the strong Huygens' principle does not
hold.  Most common examples of the latter are propagation of acoustic
waves in two dimensions or in a dissipative medium and propagation of
viscoelastic waves.  Here, we will give a brief introduction to
time-domain boundary integral equations -- for more background information see
\cite{costabel04}.

Let $\Omega \subset \mathbb{R}^d$, $d = 2,3$,  be a bounded Lipschitz domain with boundary $\Gamma = \partial \Omega$ and let $\Omega^c = \mathbb{R}^d \setminus \overline{\Omega}$ be its complement.   Let $u$ be a causal solution of the dissipative wave equation in $\mathbb{R}^d \setminus \Gamma$
\begin{equation}\label{eq:wave_eqn}
\begin{aligned}
  \partial_t^2 u(t,x) + \alpha \partial_t u(t,x)- \Delta u(t,x) &= 0, & t &\in [0,T], x \in \mathbb{R}^d \setminus \Gamma \\
u(t,x) &= g(t,x), & t &\in [0,T], x \in \Gamma \\
   u(0,x) = \partial_t u(0,x) &= 0, &  x &\in \mathbb{R}^d \setminus \Gamma,
\end{aligned}
\end{equation}
where $\alpha \geq 0$ is a constant and $g(t,x)$ is a given boundary data.

The solution of (\ref{eq:wave_eqn}) can be represented as a {\em single layer potential}
\begin{equation}
  \label{eq:slp}
  u(t,x) = \int_0^t \int_\Gamma k(t-\tau,|x-y|) \phi(\tau,y) d\Gamma_y d\tau,
\end{equation}
where the boundary density $\phi$ is the unique solution of the boundary integral equation, see  \cite{Lub94}: Find $\phi$ such that
\begin{equation}
  \label{eq:slp_eq}
  g(t,x) = \int_0^t \int_\Gamma k(t-\tau,|x-y|) \phi(\tau,y) d\Gamma_y d\tau,
\quad \text{for all } (t,x) \in [0,T] \times \Gamma.
\end{equation}

Explicit formulas for the kernel function are complicated, see \cite{costabel04}, and even unavailable in the literature for two dimensions and $\alpha > 0$.  Nevertheless, the Laplace transform of the kernel function
\[
K(\lambda,d) = \left(\mathscr{L} k\right) (\lambda,d) = \int_0^\infty e^{-\lambda t} k(t,d) dt,
\]
the so-called transfer function, is easily written down explicitly:
\begin{equation}
  \label{eq:Ks_wave}
K(\lambda,d)=
\begin{cases}
\frac1{2\pi}K_0\left(\sqrt{\lambda^2+\alpha \lambda}d\right) & \text{in two dimensions}\\
\frac{e^{-\sqrt{\lambda^2+\alpha \lambda}d}}{4\pi d} &  \text{in three dimensions},
\end{cases}
\end{equation}
where $K_0(\cdot)$ is a modified Bessel function.  The transfer function satisfies the bounds (\ref{eq:bound_nsect}) and (\ref{eq:bound_sect}) as first noticed in \cite{BanG}.

\begin{lemma}\label{lemma:wave_kern}
  For a fixed $d > 0$, the function $K(s,d)$ given in (\ref{eq:Ks_wave}) satisfies   (\ref{eq:bound_sect}) with
\[
\mu =
\begin{cases}
  -1/2 & \text{in two dimensions}\\
 0  &  \text{in three dimensions}.
\end{cases}
\]
\end{lemma}
\begin{proof}
  The bound (\ref{eq:bound_sect}) is obvious in the 3D case and for the two dimensional case follows from the asymptotic expansions for large arguments of $K_0(z)$, see \cite{AbrS}.  In \cite[Lemma~3.2]{BanG} the proof of the result is given for the two dimensional case and $\alpha =0$ and for three dimensional case and general $\alpha \geq 0$. The remaining case is a consequence of these results as shown next
\begin{align*}
  \left|e^{\lambda d} K_0(\sqrt{\lambda^2+\alpha \lambda}d)\right| &= \left|e^{\lambda d-\sqrt{\lambda^2+\alpha \lambda}d}\right| \,
\left|e^{\sqrt{\lambda^2+\alpha \lambda}d} K_0(\sqrt{\lambda^2+\alpha \lambda}d)\right|\\
&\leq C(\sigma,d) \left|\sqrt{\lambda^2+\alpha \lambda}\right|^{-1/2}\\
&\leq C(\sigma,d) \left|\lambda\right|^{-1/2}.
\end{align*}
\end{proof}

This now allows us to discretize the time convolution in (\ref{eq:slp_eq}) using convolution quadrature:
\begin{equation}
  \label{eq:slp_eq_cq}
  g_n(x) = \sum_{j = 0}^n \int_\Gamma \omegav_{n-j}(|x-y|) \phiv_j(y) d\Gamma_y, \qquad n = 0,1,\dots, T/h = N,
\end{equation}
where the weights $\omegav_{n-j}(|x|)$ are determined from the kernel
function $K(s,|x|)$ and a choice of linear multistep or Runge-Kutta
method in the same way as in the previous sections. To simplify the
description of the method, we will confine ourselves to single stage
RK methods, i.e., the backward Euler method. Numerical results will be
for multistage RK based discretization, details of implementing these
can be found in \cite{Ban10}. Note that we have used the notation
$g_n(x) = g(t_n,x)$ (and $\phi_n(x) \approx \phi(t_n,x)$) above; in
the case of multistage method these would be vectors of size $s$ as
before. The convergence of such an approximation of the integral
operators has been first analyzed in \cite{Lub94} for linear multistep
methods and then in \cite{BanL,BanLM} for Runge-Kutta methods.

To obtain a fully discrete system we need to discretize
(\ref{eq:slp_eq_cq}) in space as well. Here we will make use of a
standard Galerkin boundary element method. In order to do this, let
$\{\Gamma_1,\Gamma_2,\dots,\Gamma_M\}$ with $\cup \overline{\Gamma_j}
= \Gamma$ be a boundary element mesh and let $S_h =
\operatorname{Span}\{b_1,b_2,\dots,b_M\}$ be a subspace of
$H^{-1/2}(\Gamma)$, in particular let it be the space of piece-wise
constant functions with the basis defined by
\[
b_i(x) =
\begin{cases}
  1 \qquad \text{if } x \in \overline{\Gamma_i},\\
  0 \qquad \text{otherwise}.
\end{cases}
\]
Writing (\ref{eq:slp_eq_cq}) in a variational form and discretizing by the Galerkin method we obtain the fully discrete system: Find $\phi_{kj}$ such that
\[
\int_\Gamma g_n(x)b_\ell(x)dx = \sum_{j = 0}^n \int_\Gamma \int_\Gamma \omega_{n-j}(|x-y|) \phi_{kj}b_k(y)b_\ell(x) d\Gamma_y d\Gamma_x,
\]
$n = 0,1,\dots, T/h = N$. It will be convenient for the later discussion to rewrite this system in a matrix notation
\begin{equation}
  \label{eq:full_disc_mat}
  \mathbf{g}_n = \sum_{j = 0}^n \mathbf{W}_{n-j} \boldsymbol{\phi}_j.
\end{equation}

The simplest way of applying the techniques developed in this paper would be to apply oblivious quadrature for times $t_n > \operatorname{diam}(\Omega)/\gamma(\xi)$. A more efficient approach would be to split the matrices $\mathbf{W}_j$ into distance classes. For example given a positive constant $d_1 < \operatorname{diam}(\Omega)$ we could split each matrix into two as follows
\begin{equation}
  \label{eq:splitting}
\left(\mathbf{W}_k^{(1)}\right)_{ij} = \left\{
\begin{array}{ll}
  (\mathbf{W}_{k})_{ij} & \text{if } \operatorname{dist}(\Gamma_i,\Gamma_j) \leq d_1,\\
0 & \text{otherwise}
\end{array}
\right.
\quad \mathbf{W}_k^{(2)} = \mathbf{W}_k - \mathbf{W}_k^{(1)}
\end{equation}
Then the above sum could be split as
\[
\mathbf{g}_n = \sum_{j = 0}^n \mathbf{W}^{(1)}_{n-j} \boldsymbol{\phi}_j
+ \sum_{j = 0}^n \mathbf{W}^{(2)}_{n-j} \boldsymbol{\phi}_j
\]
and the new method applied to the first sum for $t_n > d_1/\gamma(\xi)$ and to the second for $t_n > \operatorname{diam}{\Omega}/\gamma(\xi)$, i.e., this way we can obtain savings earlier for the computation of the first sum.

\begin{remark}
In this example it is important that the quadrature used to compute
the weights $\mathbf{W}_j^{(1)}$ or $\mathbf{W}_j^{(2)}$ is valid for
a range of distances $d$. To compute the relevant parameters we
proceed as explained in Section~\ref{sec:quad} taking $d
= d_1$ or $d = \operatorname{diam}(\Omega)$. As shown in
Corollary~\ref{coro:para}, these parameters are then valid also for
any $\tilde d < d$.
\end{remark}

With this the stage is set for applying the oblivious quadrature
techniques of the previous section to the setting of time-domain
boundary integral equations described above. The algorithm is adapted to solve convolution integrals such as the one in \eqref{eq:slp_eq} in the same way as explained in \cite[Section 4.2]{SchLoLu06}. From \eqref{eq:full_disc_mat} we see that the scheme is implicit in the vector of stages $\phiv_n$. The fast and oblivious algorithm is then applied to deal with the evaluation of the history term in the right hand side of the linear system
$$
\mathbf{W}_{0} \boldsymbol{\phi}_n= \mathbf{g}_n - \sum_{j = 0}^{n-1} \mathbf{W}_{n-j} \boldsymbol{\phi}_j.
$$
Results of numerical
experiments are given in the next section.

\section{Numerical experiments}

\subsection{Wave equation in two dimensions}
Let $\Omega \subset \mathbb{R}^2$ be a disk with radius one.  We
solve~\eqref{eq:wave_eqn} with $\alpha = 0$ and $g(x,t) = t^4
\mathrm{e}^{-2t}$. We integrate from $t=0$ to $T=40$ with step-size $h =
T/400$, i.e. $N=400$, and discretize in space with equally large
patches of size $\approx 2\pi/100$, i.e. $M=100$. Because of the symmetries of the circle, see \cite{SauV14}, the solution is space independent.  It is shown in Figure~\ref{fig:solbem} on the time interval
$t\in[0,40]$.
\begin{figure}
  \centering
  \includegraphics [width=4in]{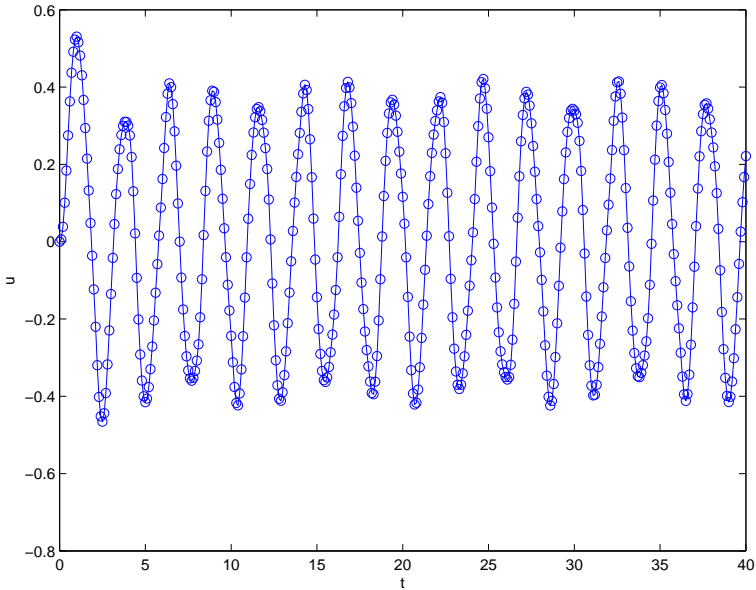}
  \caption{\small Solution of~\eqref{eq:wave_eqn} for $t\in (0,40)$, $x=(0,1)$.}
  \label{fig:solbem}
\end{figure}
The splitting of the weights into two distance classes is done by
setting the parameter $d_1 = \sqrt 2$ in
Equation~\eqref{eq:splitting}, such that the entries of $\mathbf{W}$
are divided into two equally large distance classes.  In all
experiments we chose the basis $B = 5$, which defines the splitting
in~\eqref{eq:timesplit} together with the offset $n_0 = \lceil
d/h/\gamma(\xi)\rceil$, where $d$ depends on the distance class and is either
$2$ or $\sqrt{2}$.


The evolution of the error for contour parameters $\alpha = 0.98$ and $b=0.33$ for different $L$ and $\gamma(\xi)$
is shown in Figure~\ref{fig:evolutionerror}. Increasing $L$ and decreasing $\gamma(\xi)$
yields more accurate results. The error is measured against a reference solution calculated with
a standard CQ algorithm described in \cite{BanS}.
\begin{figure}
  \centering
  \includegraphics [width=4in]{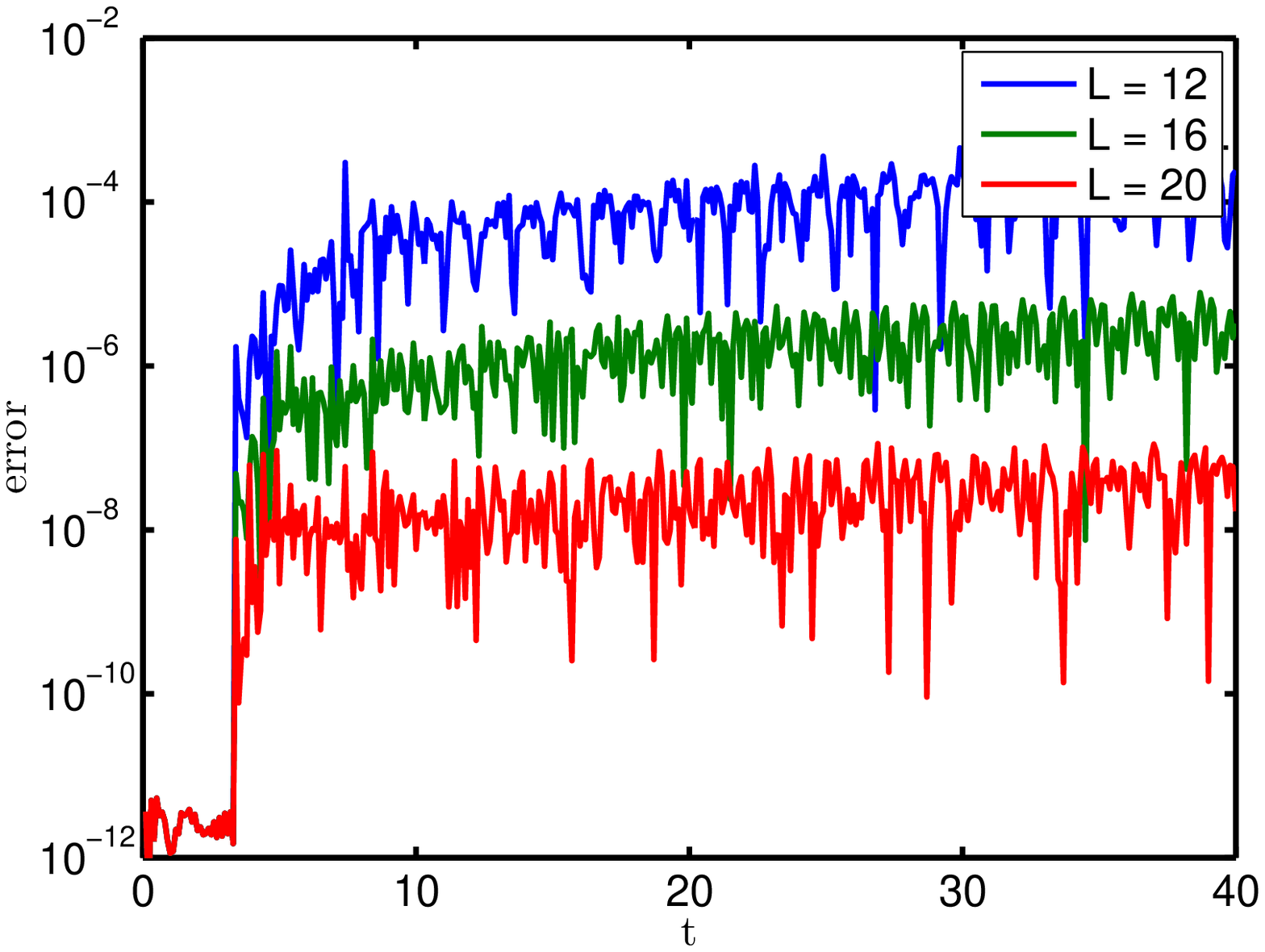}
  \includegraphics [width=4in]{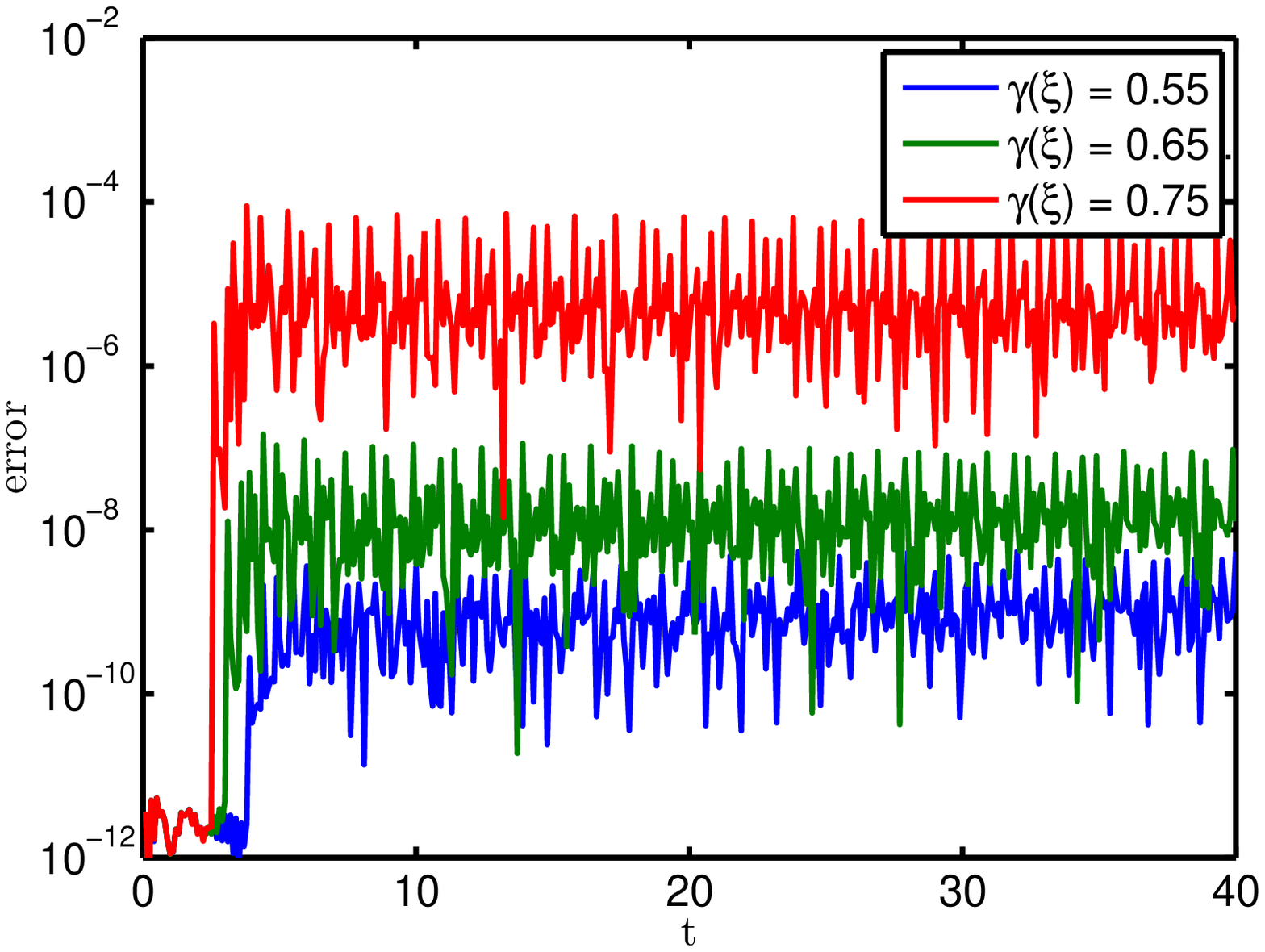}
  \caption{\small Evolution of the error:
    For different $L$ with fixed $\gamma(\xi) = 0.6$ (top) and
    for different $\gamma(\xi)$ and fixed $L=26$ (bottom).}
  \label{fig:evolutionerror}
\end{figure}
Convergence in the number of quadrature nodes $L$ for different $\gamma(\xi)$
is shown in Figure~\ref{fig:convergenceCont}. The error here is measured in the sup norm. Only if $\gamma(\xi)$ is small enough increasing $L$ improves the result. In this case
we observe exponential convergence in $L$ in agreement with Theorem~\ref{thm:main}.
\begin{figure}
  \centering

  \caption{\small Convergence in the number of contour quadrature nodes $L$.}
  \label{fig:convergenceCont}
\end{figure}
The choice of $\alpha$ and $b$ is done experimentally.

\begin{remark}
A better choice of parameters might be feasible including the angle $\alpha$ in the optimization process and eliminating $b$, as it is done in \cite{WeiTre} for the numerical inversion of the Laplace transform. In our example we have tested the parameters from \cite{WeiTre}  on a purely experimental basis and the convergence rates are actually better. However we point out that the theory in \cite{WeiTre} does not apply to our situation. Another issue with the parameters from \cite{WeiTre} is that the propagation of the errors in the evaluations of the Laplace transform $K$ is not under control, as can be observed in Figure~\ref{fig:convergenceConttw} when the error curves reach the accuracy of the reference solution, about $10^{-10}$. A further study of the optimal parameters in the context of 2D and damped wave equations might be the topic of future research.
\begin{figure}
  \centering
  \includegraphics [width=4in]{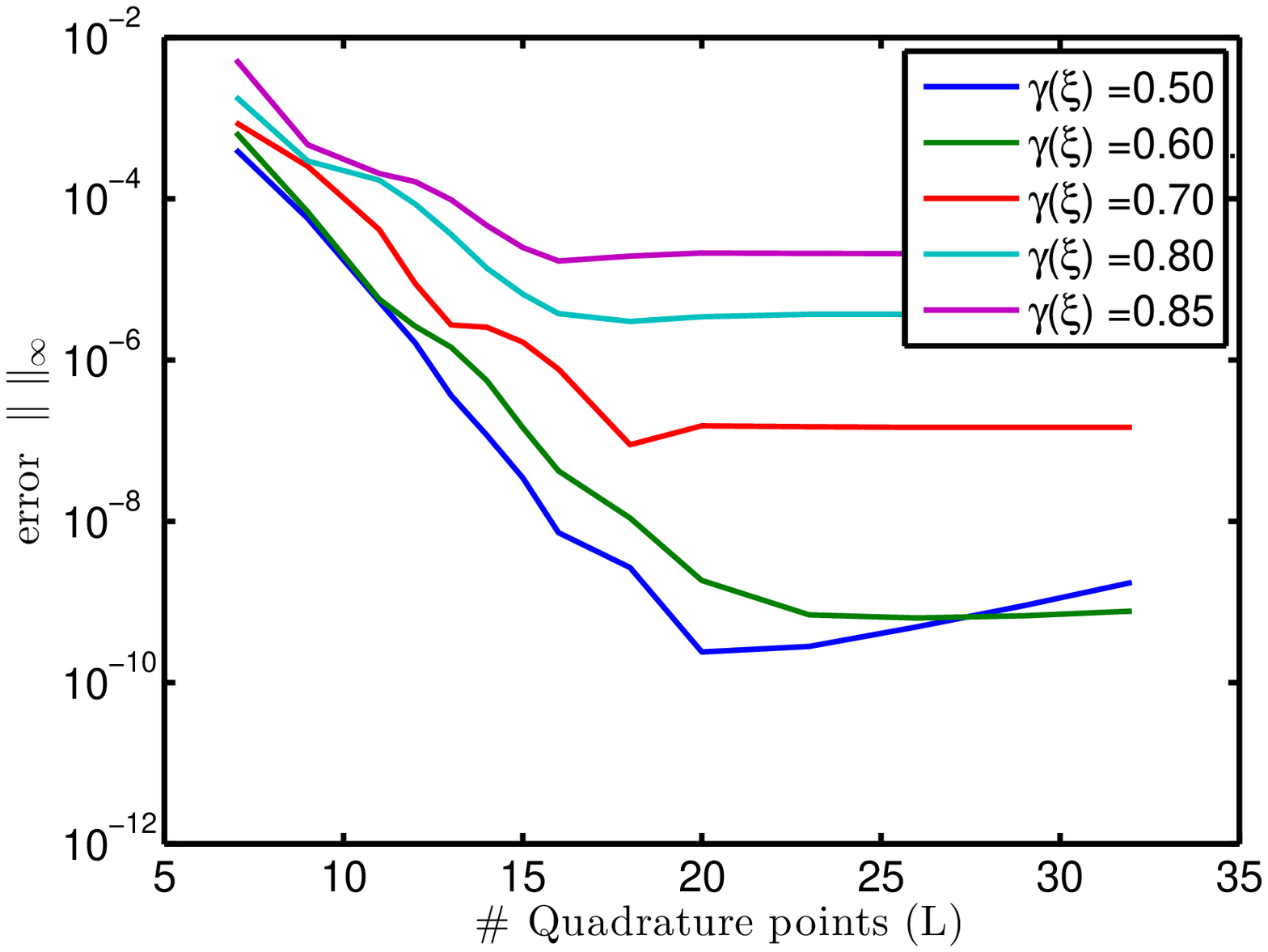}
  \caption{\small Analogous to Figure~\ref{fig:convergenceCont} by using the parameters in \cite{WeiTre}}
  \label{fig:convergenceConttw}
\end{figure}
\end{remark}

\section{Appendix}

We modify the proof given in \cite{JavT} to show the following result.

\begin{lemma}\label{lemma:trapez}
  Let $f$ be analytic and bounded as $|f(z)| \leq M$ for $z \in R_{\tau_0} = \{w \in \mathbb{C} : -a-\tau_0 \leq \Re w \leq a+\tau_0$, $-b < \Im w < b\}$ and some $\tau_0 > 0$. Further, let
\[
I = \int_{-a}^a f(x) dx, \qquad I_L = \frac{a}{L}\sum_{k = -L}^Lf(x_k),
\]
where $x_k = k\tau$, $\tau = a/L$ and $0 < \tau \leq \tau_0$. Then
\[
|I-I_L| \leq \frac{2M}{e^{2\pi b/\tau}-1}+ \tfrac{\log 2}{\pi} \tau\sup_{y \in [-b,b]}|g_{\tau/2}(y)|,
\]
where $g_{\tau/2}(y) = f(a+\tau/2+\mi y)-f(-a-\tau/2+\mi y)$.
\end{lemma}
\begin{proof}
  Let $\Gamma$ be the boundary of the rectangle $R_{\tau/2} \subset
R_{\tau_0}$ with $\Gamma_1$ the top and bottom sides and $\Gamma_2$
the left and right sides of the rectangle. Using residue calculus
\[
I_L = \int_\Gamma f(z) (2\mi)^{-1}\cot(\pi z/\tau)dz
\]
whereas using the analyticity of $f$
\[
I = \int_\Gamma f(z)u(z) dz
\]
where
\[
u(z) = \left\{
  \begin{array}{rl}
    -\frac12 & \Im z > 0,\vspace{.2cm}\\
    \frac12 & \Im z < 0.
  \end{array}
\right.
\]
Combining these two formulas we have
\[
I_L-I = \int_\Gamma f(z)S(z) dz
\]
where $S(z) = (2\mi)^{-1} \cot(\pi z/\tau)-u(z)$ or when simplified
\[
S(z) = \left\{
  \begin{array}{ll}
    \frac1{1-e^{-2\mi\pi z/\tau}} & \Im z > 0,\vspace{.2cm}\\
    \frac1{e^{2\mi\pi z/\tau}-1} & \Im z < 0.\\
  \end{array}
\right.
\]
We now split the error into two terms
\[
I_L-I = I_1+I_2
\]
corresponding to $\Gamma_1$ and $\Gamma_2$. The first term is easily bounded to give the first term in the above error estimate. Noticing that $S(\pm a\pm\tau/2+\mi y) = S(\tau/2+\mi y)$ we see that
\[
\begin{split}
|I_2| &= \left|\int_{-b}^b g_{\tau/2}(y) S(\tau/2+\mi y) dy\right| \leq \sup_{y \in [-b,b]}|g_{\tau/2}(y)|\int_{-b}^b|S(\tau/2+\mi y)| dy\\
&\leq \sup_{y \in [-b,b]}|g_{\tau/2}(y)|\int_{-\infty}^{\infty}|S(\tau/2+\mi y)| dy
= \tfrac{\log 2}{\pi} \tau \sup_{y \in [-b,b]}|g_{\tau/2}(y)|,
\end{split}
\]
where we have used
\[
\begin{split}
\int_0^\infty |S(\tau/2+\mi y)| dy &= \int_0^\infty \frac{1}{1+e^{2\pi y/\tau}}dy\\
& = \tau \int_0^\infty \frac{1}{1+e^{2\pi u}}du =  \tfrac{\log 2}{2\pi}\tau
\end{split}
\]
and similarly $\int_{-\infty}^0 |S(\tau/2+\mi y)| dy = \frac{\log 2}{2\pi}\tau$.
\end{proof}

\def\cprime{$'$}


\begin{thebibliography}{10}

\bibitem{AbrS}
M.~Abramowitz and I.~A. Stegun, editors.
\newblock {\em Handbook of mathematical functions with formulas, graphs, and
  mathematical tables}.
\newblock Dover Publications Inc., New York, 1992.

\bibitem{stas_maxwell}
J.~Ballani, L.~Banjai, S.~Sauter, and A.~Veit.
\newblock Numerical solution of exterior {M}axwell problems by {G}alerkin {BEM}
  and {R}unge-{K}utta convolution quadrature.
\newblock {\em Numer. Math.}, 123(4):643--670, 2013.

\bibitem{BamH}
A.~Bamberger and T.~Ha-{D}uong.
\newblock Formulation variationelle espace-temps pour le calcul par potentiel
  retard{\'{e}} d'une onde acoustique.
\newblock {\em Math. Meth. Appl. Sci.}, 8:405--435, 1986.

\bibitem{Ban10}
L.~Banjai.
\newblock Multistep and multistage convolution quadrature for the wave
  equation: Algorithms and experiments.
\newblock {\em SIAM J. Sci. Comput.}, 32(5):2964--2994, 2010.

\bibitem{BanG}
L.~Banjai and V.~Gruhne.
\newblock Efficient long time computations of time-domain boundary integrals
  for 2d and dissipative wave equation.
\newblock {\em J. Comput. Appl. Math.}, 235(14):4207--4220, 2011.

\bibitem{BanKfast}
L.~Banjai and M.~Kachanovska.
\newblock Sparsity of {R}unge--{K}utta convolution weights for the
  three-dimensional wave equation.
\newblock {\em BIT}, 54(4):901--936, 2014.

\bibitem{BanL}
L.~Banjai and C.~Lubich.
\newblock An error analysis of {R}unge-{K}utta convolution quadrature.
\newblock {\em BIT}, 51(3):483--496, 2011.

\bibitem{BanLM}
L.~Banjai, C.~Lubich, and J.~M. Melenk.
\newblock Runge-{K}utta convolution quadrature for operators arising in wave
  propagation.
\newblock {\em Numer. Math.}, 119(1):1--20, 2011.

\bibitem{BanS}
L.~Banjai and S.~Sauter.
\newblock Rapid solution of the wave equation in unbounded domains.
\newblock {\em SIAM J. Numer. Anal.}, 47(1):227--249, 2008/09.

\bibitem{Mansur2d}
J.~A.~M. Carrer, W.~L.~A. Pereira, and W.~J. Mansur.
\newblock Two-dimensional elastodynamics by the time-domain boundary element
  method: {L}agrange interpolation strategy in time integration.
\newblock {\em Eng. Anal. Bound. Elem.}, 36(7):1164--1172, 2012.

\bibitem{costabel04}
M.~Costabel.
\newblock Time-dependent problems with the boundary integral equation method.
\newblock 1, Fundamentals, 2005.

\bibitem{DavD13}
P.~J. Davies and D.~B. Duncan.
\newblock Convolution-in-time approximations of time domain boundary integral
  equations.
\newblock {\em SIAM J. Sci. Comput.}, 35(1):B43--B61, 2013.

\bibitem{DavD14}
P.~J. Davies and D.~B. Duncan.
\newblock Convolution spline approximations for time domain boundary integral
  equations.
\newblock {\em J. Integral Equations Appl.}, 26(3):369--412, 2014.

\bibitem{HaWII}
E.~Hairer and G.~Wanner.
\newblock {\em Solving ordinary differential equations. {II}}, volume~14 of
  {\em Springer Series in Computational Mathematics}.
\newblock Springer-Verlag, Berlin, second edition, 1996.
\newblock Stiff and differential-algebraic problems.

\bibitem{JavT}
M.~Javed and L.~N. Trefethen.
\newblock A trapezoidal rule error bound unifying the {E}uler-{M}aclaurin
  formula and geometric convergence for periodic functions.
\newblock {\em Proc. R. Soc. Lond. Ser. A Math. Phys. Eng. Sci.},
  470(2161):20130571, 9, 2014.

\bibitem{LalS}
A.~R. Laliena and F.-J. Sayas.
\newblock Theoretical aspects of the application of convolution quadrature to
  scattering of acoustic waves.
\newblock {\em Numer. Math.}, 112(4):637--678, 2009.

\bibitem{LopLPS}
M.~L{\'o}pez-Fern{\'a}ndez, C.~Lubich, C.~Palencia, and A.~Sch{\"a}dle.
\newblock Fast {R}unge-{K}utta approximation of inhomogeneous parabolic
  equations.
\newblock {\em Numer. Math.}, 102(2):277--291, 2005.

\bibitem{LoLuSch08}
M.~L{\'o}pez-Fern{\'a}ndez, C.~Lubich, and A.~Sch{\"a}dle.
\newblock Adaptive, fast, and oblivious convolution in evolution equations with
  memory.
\newblock {\em SIAM J. Sci. Comput.}, 30(2):1015--1037, 2008.

\bibitem{LoPa}
M.~L{\'o}pez-Fern{\'a}ndez and C.~Palencia.
\newblock On the numerical inversion of the {L}aplace transform of certain
  holomorphic mappings.
\newblock {\em Appl. Numer. Math.}, 51(2-3):289--303, 2004.

\bibitem{LoPaScha}
M.~L{\'o}pez-Fern{\'a}ndez, C.~Palencia, and A.~Sch{\"a}dle.
\newblock A spectral order method for inverting sectorial {L}aplace transforms.
\newblock {\em SIAM J. Numer. Anal.}, 44(3):1332--1350 (electronic), 2006.

\bibitem{LopS}
M.~L\'opez-Fern\'andez and S.~Sauter.
\newblock Generalized convolution quadrature with variable time stepping.
\newblock {\em IMA J. Numer. Anal.}, 33(4):1156--1175, 2013.

\bibitem{LopS2}
M.~L\'opez-Fern\'andez and S.~Sauter.
\newblock Generalized convolution quadrature with variable time stepping. part
  ii: algorithm and numerical results.
\newblock {\em Preprint}, 2013.

\bibitem{Lub94}
C.~Lubich.
\newblock On the multistep time discretization of linear initial-boundary value
  problems and their boundary integral equations.
\newblock {\em Numer. Math.}, 67:365--389, 1994.

\bibitem{LubO93}
C.~Lubich and A.~Ostermann.
\newblock Runge-{K}utta methods for parabolic equations and convolution
  quadrature.
\newblock {\em Math. Comp.}, 60(201):105--131, 1993.

\bibitem{LubS}
C.~Lubich and A.~Sch{\"a}dle.
\newblock Fast convolution for nonreflecting boundary conditions.
\newblock {\em SIAM J. Sci. Comput.}, 24(1):161--182, 2002.

\bibitem{CheMWW}
X.~W. Q.~Chen, P.~Monk and D.~Weile.
\newblock Analysis of convolution quadrature applied to the time-domain
  electric field integral equation.
\newblock {\em Commun. Comput. Phys.}, 11:383--399, 2012.

\bibitem{SauV13}
S.~Sauter and A.~Veit.
\newblock A {G}alerkin method for retarded boundary integral equations with
  smooth and compactly supported temporal basis functions.
\newblock {\em Numer. Math.}, 123(1):145--176, 2013.

\bibitem{SauV14}
S.~Sauter and A.~Veit.
\newblock Retarded boundary integral equations on the sphere: exact and
  numerical solution.
\newblock {\em IMA J. Numer. Anal.}, 34(2):675--699, 2014.

\bibitem{Say13}
F.-J. Sayas.
\newblock Energy estimates for {G}alerkin semidiscretizations of time domain
  boundary integral equations.
\newblock {\em Numer. Math.}, 124(1):121--149, 2013.

\bibitem{SchLoLu06}
A.~Sch{\"a}dle, M.~L{\'o}pez-Fern{\'a}ndez, and C.~Lubich.
\newblock Fast and oblivious convolution quadrature.
\newblock {\em SIAM J. Sci. Comput.}, 28(2):421--438, 2006.


\bibitem{MR3092588}
H.~A. {\"U}lk{\"u}, H.~Ba{\u{g}}c{\i}, and E.~Michielssen.
\newblock Marching on-in-time solution of the time domain magnetic field
  integral equation using a predictor-corrector scheme.
\newblock {\em IEEE Trans. Antennas and Propagation}, 61(8):4120--4131, 2013.

\bibitem{MR3055890}
F.~Vald{\'e}s, M.~Ghaffari-Miab, F.~P. Andriulli, K.~Cools, and E.~Michielssen.
\newblock High-order {C}alder\'on preconditioned time domain integral equation
  solvers.
\newblock {\em IEEE Trans. Antennas and Propagation}, 61(5):2570--2588, 2013.


\bibitem{WeiTre}
J.A.C~Weideman and L.N.~Trefethen.
\newblock Parabolic and hyperbolic contours for computing the Bromwich integral.
\newblock {\em Math. Comput.}, 76(2):1341-1356 , 2007.


\end{thebibliography}

\end{document}